\documentclass[12pt]{amsart}
\usepackage{amscd,amssymb}
\usepackage{graphicx,enumerate}

\usepackage{color}
\usepackage[arrow,matrix,arrow,ps,color,line,curve,frame]{xy}
\usepackage{pgf, tikz}
\usepackage{url}
\usepackage{enumerate}

\makeatletter
\def\url@leostyle{%
  \@ifundefined{selectfont}{\def\UrlFont{\sf}}{\def\UrlFont{\small\ttfamily}}}
\makeatother
\usetikzlibrary{matrix,arrows}

\usepackage[colorinlistoftodos,bordercolor=red,backgroundcolor=red!20,linecolor=red,textsize=scriptsize]{todonotes}

\let\oldlabel=\label

\def\prellabel{\marginparsep=1em
    \def\label##1{\oldlabel{##1}\ifmmode\else\ifinner\else
         \marginpar{{\footnotesize\ \\ \tt
                    ##1}}\fi\fi}}

\def\inte{\operatorname{relint}}
\def\conv{\operatorname{conv}}
\def\vertex{\operatorname{vert}}
\def\rank{\operatorname{rank}}
\def\Im{\operatorname{Im}}
\def\Hom{\operatorname{Hom}}

\def\Aff{\operatorname{Aff}}

\def\o{\operatorname{o}}
\def\ev{\operatorname{ev}}
\def\op{\operatorname{op}}
\def\vol{\operatorname{vol}}
\def\coker{\operatorname{coker}}

\def\codim{\operatorname{codim}}
\def\surj{\operatorname{surj}}
\def\inj{\operatorname{inj}}

\def\FCONV{\mathbf{FConv}}
\def\F{\mathbf{F}}

\def\cD{{\mathcal D}}
\def\cN{{\mathcal N}}

\def\cC{{\mathcal C}}

\def\RR{{\mathbb R}}

\def\FF{{\mathbb F}}

\def\SETS{\mathbf{Sets}}
\def\POL{\mathbf{Pol}}
\def\COMP{\mathbf{Comp}}
\def\CONV{\mathbf{Conv}}

\def\VECT{\mathbf{Vect}}

\let\phi=\varphi
\let\theta=\vartheta
\let\epsilon=\varepsilon

\advance\headheight1.15pt

\newtheorem{lemma}{Lemma}[section]
\newtheorem{corollary}[lemma]{Corollary}
\newtheorem{theorem}[lemma]{Theorem}
\newtheorem{proposition}[lemma]{Proposition}

\theoremstyle{definition}
\newtheorem{definition}[lemma]{Definition}
\newtheorem{remark}[lemma]{Remark}
\newtheorem{example}[lemma]{Example}

\newtheorem{problem}[lemma]{Problem}

\begin{document}

\title[Affine-compact functors]{Affine-compact functors}

\begin{abstract}
Several well known polytopal constuctions are examined from the functorial point of view. A naive analogy between the Billera-Sturmfels fiber polytope and the abelian kernel is disproved by an infinite explicit series of polytopes. A correct functorial formula is provided in terms of an affine-compact substitute of the abelian kernel. The dual cokernel object is almost always the natural affine projection. The Mond-Smith-van Straten space of sandwiched simplices, useful in stochastic factorizations, leads to a different kind of affine-compact functors and new challenges in polytope theory.
\end{abstract}

\author{Joseph Gubeladze}

\address{Department of Mathematics\\
         San Francisco State University\\
         1600 Holloway Ave.\\
         San Francisco, CA 94132, USA}
\email{soso@sfsu.edu}

\thanks{Supported by U.S. NSF grant DMS 1301487 and Georgian NSF grant DI/16/5-103/12}



\maketitle

\section{Introduction}\label{Intro}

The role of representable functors in algebraic geometry and topology is well known. In this work, we initiate a functorial approach to various important convex and polytopal constructions, with similar emphasize on representable functors.  

On the one hand, the category $\POL$ of \emph{convex polytopes} and \emph{affine maps} is `too linear' for the two mentioned disciplines. But, on the other hand, the rich combinatorial structure it carries makes $\POL$ the backbone of geometric and, to a large extent, algebraic combinatorics. Methods and techniques from algebraic topology and geomerty are often used to solve important open problems in combinatorics. A natural question is whether $\POL$ itself can be subjected to a categorial/homological analysis. Put another way, one can ask whether (i) the representable functors of well-known polytopal objects have expected properties within the  category of functors defined on $\POL$, and (ii) natural  polytopal correspondences are representable functors, leading to new geometric objects. A first indication that these are meaningful questions is the initial observation that $\POL$ is a self-enriched symmetric monoidal category in a natural way; see Section \ref{preliminaries}. For any $P,Q\in\POL$, the facets of $\Hom(P,Q)$ and vertices of $P\otimes Q$ are readily described. However, determination of the vertices of $\Hom(P,Q)$ and facets of $P\otimes Q$ is a real challenge and partial progress in this direction is accomplished in \cite{Hompolytopes,Vertices}.

This work represents the next natural step beyond $\Hom$ and $\otimes$. Namely, 

\medskip\noindent$\centerdot$ In Section \ref{Kernel}, we examine the Billera-Sturmfels \emph{fiber polytope} $\Sigma f$ \cite{Fiber} from the functorial perspective. This polytope is the average over the fibers of a map $f$ in $\POL$ and defined in terms of the Minkowski integral. Fiber polytopes are generalizations of the Gel'fand-Kapranov-Zelevinsky \emph{secondary polytopes} \cite{Secondary2} and have many applications, especially in triangulation theory \cite[Ch.9]{Triangulations}. They are reminiscent of the kernels of linear maps, but only informally as the category $\POL$ is far from being abelian -- it even lacks a $0$ object. Still, one can ask whether for $R$ and $f:P\to Q$ in $\POL$ the polytopes $\Hom(R,\Sigma f)$ and $\Sigma\Hom(R,f)$ are isomorphic, mimicking the functorial isomorphism $\Hom(-,\ker\alpha)\cong\ker\Hom(-,\alpha)$ in the abelian setting. In Theorem \ref{main1} we provide an infinite series of polytopal counterexamples.

\medskip\noindent$\centerdot$ In Section \ref{Conker}, we develop an affine-compact version of the linear kernel for the more general category $\CONV$ of all \emph{convex compact sets} and affine maps. It leads to the correct version (Theorem \ref{kaf=fiber}) of the naive fiber equality, which was disproved in Section \ref{Kernel}. The \emph{affine-compact kernel} is preceded in Section \ref{Definitions} by a similar analysis of the Minkowski sum. Even if one wants to work exclusively with $\POL$, the limit sets enter the picture via the proof of the central Lemma \ref{minkowskisum}. This makes the passage from $\POL$ to $\CONV$ even more natural.

\medskip\noindent$\centerdot$ In Section \ref{Cokernel}, we show that the dual concept of the affine-compact cokernel is less geometrically meaningful: for a map $f:X\to Y$ in $\CONV$,  it is (almost always) the linear projection of $Y$ along the affine hull of $f(X)$.

\medskip\noindent$\centerdot$ Section \ref{Difference} represents a more radical departure from the linear setup. Motivated by the \emph{space of sandwiched simplices}, which was introduced by Mond-Smith-van Straten \cite{Sandwich} for modeling stochastic factorizations, we define a pair of functors: \emph{sandwiching} and \emph{complementing}. For a map $f:Y\to X$, the first functor makes $f(Y)$ a necessary target, like $0$ in the abelian situation, and the other makes the interior of $f(Y)$ an impossible target, a fusion of the topological quotient and affine maps. We observe that these functors are still \emph{affine-compact} but with values beyond $\CONV$. Unlike the sandwiches, the topological behavior of the complementing functor is transparent, and there is a complementarity between the functors (Theorem \ref{comp-quotient}). 

\medskip\noindent$\centerdot$ In Section \ref{Challenges}, we discuss new challenges in polytope theory the functorial approach leads to. 

\medskip\noindent\emph{Acknowledgment.} I thank (i) the referee for many constructive suggestions, which greatly improved the exposition, and for spotting a number of inaccuracies,  (ii) Tristram Bogart, from whom I learned about the question, attributed to someone else, whether $\Sigma\Hom(R,f)\cong\Hom(R,\Sigma f)$, and (iii) Timmy Chan for Figure 1.

\section{Categorial preliminaries}\label{preliminaries}

Traditionally, category theory is not necessarily in the toolkit of those working in convex geometry or polytope theory. In this section, we give a brief informal introduction to some basic terminology, necessary for us in this work. For a formal treatment we refer the reader to (i) the classics \cite{Categories} for the standard material on categories, and (ii) \cite{Enriched} for the enriched context.

\subsection{Polytopes and convex sets} Our references for basic facts on polytopes are \cite[Ch.1]{Kripo} and \cite{Polytopes}. 

A \emph{vector space} will always mean a finite-dimensional real vector space. Our polytopes are assumed to be \emph{convex}. An \emph{affine map} between two subsets of vector spaces are the maps, respecting barycentric coordinates. 

For a subset $X$ of a vector space, the \emph{convex} and \emph{affine hulls} will be denoted, correspondingly, by $\conv(X)$ and $\Aff(X)$. For a convex set $X$, by $\inte(X)$ we denote the \emph{relative interior} of $X$. The \emph{boundary} of $X$ is $\partial X=X\setminus\inte(X)$. For a polytope $P$, the set of its \emph{vertices}, that of \emph{facets}, and the \emph{normal fan} will be denoted by $\vertex(P)$, $\FF(P)$, and $\cN(P)$, respectively.

For an affine map $f:X\to Y$ between convex compact sets we put $\codim f=\dim X-\dim(f(X))$.

All further terminology and notation will be introduced in the text.

\subsection{Representable functors}
The categories we will be working with are: 
\begin{enumerate}[\rm(i)]
\item
$\SETS$ -- sets and maps,
\item $\VECT$ -- vector spaces and linear maps,
\item $\POL$ -- polytopes and affine maps,
\item $\CONV$ -- convex compact sets in vector spaces and affine maps,
\item $\COMP$ -- general compact subsets of vector spaces and affine maps,
\item Posets, viewed as categories (in Section \ref{Difference}).
\end{enumerate}

For a category $\cC$ and objects $a,b\in\cC$, we write $a\cong b$ if $a$ and $b$ are isomorphic. The set of morphisms $a\to b$ will be denoted by $\Hom_{\cC}(a,b)$, or just $\Hom(a,b)$ when there is no ambiguity. For $\cC=\SETS$, $\VECT$, $\POL$, or $\CONV$, the set $\Hom_{\cC}(a,b)$ is naturally an object of $\cC$. This is obvious when $\cC$ is $\SETS$ or $\VECT$; when $\cC=\POL$ this observation is the starting point of \cite{Hompolytopes}; the case $\cC=\CONV$ is shown as follows: when $a=\Delta$ is a simplex of dimension $d$ and $Y$ is an arbitrary convex set then $\Hom(\Delta,Y)\cong Y^{d+1}$ or, equivalently, any map from the vertices of $\Delta$ to $Y$ uniquely extends to an affine map $\Delta\to Y$, and for general $X\in\CONV$ we have
\begin{align*}
\Hom(X,Y)=\bigcap_{
{\tiny
\begin{matrix}
&\Delta\subset X\\
&\dim\Delta=\dim X\\
\end{matrix}
}}
\Hom(\Delta,Y),
\end{align*}
where the intersection is taken in the affine space of affine maps $\Aff(X)\to\Aff(Y)$.

For a category $\mathcal C$, the \emph{dual category} $\mathcal C^{\op}$ is the category with the same objects, where the direction of morphisms have been formally reversed. In $\mathcal C^{\op}$, the composition of two morphisms is the reversed copy of the composition of the corresponding morphisms in $\mathcal C$.

For two categories $\mathcal C$ and $\mathcal D$, a \emph{covariant functor} $F:\mathcal C\to\mathcal D$ is an object-to-object and morphism-to-morphism correspondence, respecting the identity morphisms and compositions. A \emph{contravariant functor} $G:\mathcal C\to\mathcal D$ is the same as a covariant functor $\mathcal C^{\op}\to\mathcal D$. Let $\mathcal C$ be $\SETS$, $\VECT$, $\POL$, or $\CONV$, and $a\in\mathcal C$ be an object. Then, for any morphism $f:b\to c$ in $\cC$, the maps
\begin{align*}
&\Hom(a,b)\to\Hom(a,c),\quad \phi\mapsto f\phi,\\
&\Hom(c,a)\to\Hom(b,a),\quad \psi\to \psi f,
\end{align*}
are morphisms in $\cC$. The resulting functors
\begin{align*}
\Hom(a,-):\mathcal C\to\mathcal C\quad\text{and}\quad\Hom(-,a):\mathcal C^{\op}\to\mathcal C
\end{align*}
are called \emph{hom-functors}.

A functor is \emph{affine-compact} if it is defined on $\POL$ or $\CONV$, evaluates in $\COMP$, and induces affine maps between the hom-sets. 

Examples of affine-compact functors are given by the hom-functors, defined on $\POL$ or $\CONV$.

A (covariant) functor $F:\cC\to\cD$ is called \emph{full} if the induced maps $\Hom_{\cC}(a,b)\to\Hom_{\cD}(F(a),F(b))$ are surjective for all $a,b\in\cC$. If the mentioned maps are injective, then $F$ is called \emph{faithful}. 

\subsection{Limits and colimits} Let $F:\cC\to\cD$ be a covariant functor. Then we have the category of \emph{co-cones over} $F$: its objects are families of morphisms of the form $\big(g_a:F(a)\to x\ |\ a\in\cC\big)$, where $x\in\cD$, for which the following triangles commute for all morphisms $f:a\to b$:
$$
\xymatrix{
&x&\\
F(a)\ar[rr]_{F(f)}\ar[ur]^{g_a}&&F(b)\ar[ul]_{g_b}\\
}
$$
A moprhism from a co-cone $\big(g_a:F(a)\to x\ |\ a\in\cC\big)$ to a co-cone $\big(h_a:F(a)\to y\ |\ a\in\cC\big)$ is a morphism $\alpha:x\to y$ in $\cD$, making the following triangles commutative for all $a\in\cC$.
$$
\xymatrix{
x\ar[rr]^\alpha&&y\\
&F(a)\ar[ul]^{g_a}\ar[ur]_{h_a}&\\
}
$$

The \emph{colimit} of $F$, denoted by $\lim\limits_{\longrightarrow}F$, is any  \emph{terminal} object of the category of co-cones, i.e., a co-cone which admits exactly one morphism from any co-cone. In particular, $\lim\limits_{\longrightarrow}F$ is defined up to isomorphism. By abusing terminology, the apex of a colimit co-cone will be also referred to as the colimit of $F$.  

Colimits of appropriate functors include:  the usual limits of monotonic bounded sequences in $\RR$, disjoint unions of sets, disjoint unions of topological spaces, quotient topological spaces, direct sums of modules, quotient modules, tensor products of commutative algebras etc. Any functor from a \emph{finite} category to $\POL$ or $\CONV$ has a colimit, i.e., $\POL$ and $\CONV$ are \emph{finitely co-complete}.

The dual notion is the \emph{limit} of a functor $F:\cC\to\cD$, which are defined in terms of the \emph{category of cones} $\big(g_a:x\to F(a)\ |\ a\in\cC\big)$, where $x\in\cD$. The \emph{limit} $\lim\limits_{\longleftarrow}F$ is then an \emph{initial} object of this category, i.e., a cone from which there is exactly one morphism to any cone. Again, the limit is defined up to isomorphism and, by abusing the terminology, the apex of a limit cone will be also called the limit of $F$.

Limits of appropriate functors include: direct products of sets or algebraic structures (as groups, modules, rings), the kernels of group, ring, or module homomorphisms.  Any functor from a \emph{finite} category to $\POL$ or $\CONV$ has a limit, i.e., $\POL$ and $\CONV$ are \emph{finitely complete}. 

Finite inverse limits in $\POL$ (resp. $\CONV$) and $\SETS$ agree in the following sense: for a finite category $\cC$ and a functor $F:\cC\to\POL$, we have the equality of sets $\iota\big(\lim\limits_{\longleftarrow} F\big)=\lim\limits_{\longleftarrow}(\iota\circ F)$, where $\iota:\POL\to\SETS$ is the identity embedding, and the same is true for $\CONV$.

A \emph{pull-back diagram} in $\SETS$, $\VECT$, $\POL$, or $\CONV$, is the limit of a functor to the corresponding category from the following category with three objects and two non-identity arrows:
$$
\xymatrix{
&\bullet\ar[d]\\
\bullet\ar[r]&\bullet
}
$$
More explicitly, a pull-back diagram in $\SETS$, $\VECT$, $\POL$, or $\CONV$, is a commutative diagram of the form:
$$
\xymatrix{
W\ar[r]\ar[d]&X\ar[d]^f\\
Y\ar[r]_g&Z
}
$$
where $W=\{(x,y)\ |\ f(x)=g(y)\}\subset X\times Y$ and the maps from $W$ are the projection maps. If $g$ is injective then $W$ is naturally identified with $f^{-1}(g(Y))$.

\subsection{Yoneda embedding}\label{Yoneda_embedding} For two categories and two covariant functors $F,G:\cC\to\cD$, a \emph{natural transformation} $\tau:F\xymatrix{\ar[r]^{\bullet}&}G$ is a system of morphisms $\tau=\{\tau_a:F(a)\to F(a)\}$, where $a$ runs over the objects of $\cC$, making the following diagrams commutative for all choices of  $f\in\Hom(a,b)$:
$$
\xymatrix{
F(a)\ar[r]^{\tau_a}\ar[d]_{F(f)}&G(b)\ar[d]^{G(f)}\\
F(b)\ar[r]_{\tau_b}\ar[r]&G(b)\\
}
$$

The covariant functors $\cC\to\cD$ and their natural transformations form the \emph{category of functors} $\cD^{\cC}$. In particular, $\cD^{\cC^{\op}}$ is the category of contravariant functors $\cC\to\cD$.

Our goal in this work is to develop polytopal analogues of the following isomorphims of functors:
\begin{align*} 
&\Hom_\VECT(-,\ker f)\cong\ker\Hom_\VECT(-,f),\\
&\Hom_\VECT(\coker f,-)\cong\ker\Hom_\VECT(f,-),\\
&\Hom_\VECT(-,\coker f)\cong\coker\Hom_\VECT(-,f),\\
\end{align*}
where the first two are standard facts on (co)limits of functors, and the third is an `abelian' phenomenon, meaning that $\VECT$ is an \emph{abelian category}.

When $\cC$ is $\SETS$, $\VECT$, $\POL$, or $\CONV$, a functor in $\cC^{\cC}$, (respectively, $\cC^{\cC^{\op}}$) is called \emph{representable} if it is isomorphic to $\Hom(a,-)$ (respectively, $\Hom(-,a)$) for some $a\in\cC$. Representable functors give a handle on general functors:

\medskip\noindent{\bf\emph{Yoneda Embedding.}} \emph{For $\cC=\SETS$, $\VECT$, $\POL$, or $\CONV$, the assignments 
\begin{align*}
&\cC^{\op}\to\cC^{\cC},\qquad a\mapsto\Hom(a,-),\\
&\cC\to\cC^{\cC^{\op}},\qquad a\mapsto\Hom(-,a),
\end{align*}
define covariant full and faithful embeddings; moreover, every functor in the target category is a colimit of representable functors.}

\medskip The standard Yoneda embedding concerns the case $\cC=\SETS$ or, more generally, when the functors evaluate in $\SETS$. It is a consequence of the \emph{Yoneda Lemma} \cite[Ch.2]{Categories}. The colimit representation claim for the other three categories follows from the same standard recipe for colimit representations, as given in \cite[Ch.3, \S7]{Categories}, where the details are worked out only for covariant functors. (Beware of the typos in the proof of \cite[Theorem III.7.1]{Categories}: $J^D$ in the definition of $M$ as well as $J$ in the diagram (1) are supposed to be $J^{\op}$.)

Our strategy for exploring similarities between $\POL$ and $\VECT$ is based on this embedding: $\POL^\POL$ is more amenable to such an analysis than the more rigid $\POL$, and general functors are in the `vicinity' of representable functors.

\subsection{Self-enriched categories} A more conceptual paradigm of the Yoneda Embedding for the categories of our interest is the \emph{enriched} context over \emph{symmetric monoidal categories}. The Yoneda Lemma in this generality is worked out in \cite[Ch.2]{Enriched}:
the four categories, mentioned in the Yoneda Embedding above, are symmetric monoidal categories in a natural way. This means that there is a \emph{bifunctor} $\otimes:\cC\times\cC\to\cC$, together with a distinguished object $I$ and natural isomorphisms $a\otimes I\cong a\cong I\otimes a$, $(a\otimes b)\otimes c\cong a\otimes(b\otimes c)$, and $a\otimes b\cong b\otimes a$ for any objects $a,b,c\in\cC$, satisfying certain coherence conditions. The monoidal product in $\SETS$ is the Cartesian product, with $I$ a singleton; in the case of $\VECT$, the object $I$ is the space $\RR$ and $\otimes$ is the tensor product of vector spaces; for $\POL$ the object $I$ is a singleton and the tensor product of polytopes is the \emph{dehomogenization} of the usual tensor product of the associated \emph{homogenization cones}; see \cite[Section 3]{Hompolytopes}. A particular realization is
$$
P\otimes Q=\conv\{(v\otimes w,v,w)\ |\ v\in\vertex(P)\ w\in\vertex(Q)\}\subset\big(V\otimes W\big)\oplus V\oplus W,
$$ 
where $P\subset V$ and $Q\subset W$ are the ambient vector spaces. Moreover, these monoidal categories are \emph{closed}, i.e., we have functorial isomorphisms
$$
\Hom(a\otimes b,c)\cong\Hom(a,\Hom(b,c))
$$
for all $a,b,c$. In other words, $\otimes$ and $\Hom$ form a pair of \emph{left and right adjoint functors}.

The fact that in $\SETS$, $\VECT$, or $\POL$, the sets $\Hom(a,b)$ are objects of the same category and the composition defines morphisms $\Hom(a,b)\otimes\Hom(b,c)\to\Hom(a,c)$ with natural coherence properties means that $\VECT$ and $\POL$ are \emph{self-enriched} symmetric monoidal categories.

Other classical examples of self-enriched symmetric monoidal categories include the following categories of modules over a commutative ring: general modules, finitely generated modules, torsion modules, free modules, projective modules. 

The concept naturally extends to \emph{categories, enriched over a symmetric monoidal category}. In particular, a category enriched over $\SETS$ is just the original definition of a category.  

Without delving into technical details, we observe that $\CONV$ is also a self-enriched symmetric monoidal category with respect to the tensor product
\begin{align*} 
X\otimes_{\CONV}Y:=&\bigcup(P\otimes_{\POL}Q\ |\ P\subset X\ \text{and}\ Q\subset Y\ \text{polytopes})=\\
&\conv\big(v\otimes w,v,w\ |\ v\in\partial X,\ w\in\partial Y\big)\subset\big(V\otimes W\big)\oplus V\oplus W.
\end{align*}

The monoidal structure of $\CONV$ extends that of $\POL$. There is another self-enriched symmetric monoidal extension of $\POL$, different from $\CONV$, whose objects are \emph{polytopal complexes} \cite{Hom-complexes}. The functors $X\otimes-$, where $X$ is an object in $\POL$ or $\CONV$, are further examples of affine-compact functors.

More background material on $\Hom$ and $\otimes$ in the category of general convex cones is found in \cite{Velasco}, which focuses on multilinear optimization. The undergraduate thesis \cite{Valby} makes a lucid reading on categorial generalities on polytopes and cones. 

\section{Minkowski sums, fibers, continuity}\label{Definitions}

Let $V\in\VECT$, $\lambda_1,\ldots,\lambda_n\in\RR_{\ge0}$, and $X_1,\ldots,X_n\subset V$ be subsets.  The corresponding \emph{Minkowski linear combination} is the subset
$$
\sum_{i=1}^n\lambda_iX_i=\big\{\sum_{i=1}^n\lambda_ix_i\ |\ x_i\in X_i,\ i=1,\ldots,n\big\}\subset V.
$$
It belongs to $\CONV$ if $X_i\in\CONV$, or to $\POL$ if $X_i\in\POL$. 

\begin{lemma}\label{basic-sum}
Let $P$ and $Q$ be polytopes in a vector space $V$.
\begin{enumerate}[\rm(a)]
\item $\cN(P+Q)$ is the common refinement of $\cN(P)$ and $\cN(Q)$. 
\item $\#\vertex(P)\le\#\vertex(P+Q)$ and $\#\FF(P)\le\#\FF(P+Q)$.
\item If $\cN(P)=\cN(Q)$, then the faces of $P+Q$ are the Minkowski sums of the pairs of corresponding faces of $P$ and $Q$.
\end{enumerate}
\end{lemma}

The part (a) is proved in \cite[Proposition 7.12]{Polytopes}, and (b,c) are easy consequences.

In the following definition, we assume that (i) $X\subset V$ and $Y\subset W$ are compact convex subsets of vector spaces, and (ii) in $X$ and $\Aff(f(X))$ we have chosen translation invariant Borel measures.

\begin{definition}\label{fiber}
Let $f:X\to Y$ be a map in $\CONV$. The set of \emph{sections} of $f$ is defined by
$\Gamma_f=\{\gamma:f(X)\to X\ |\ \gamma\ \text{is Borel measurable and}\ f\circ\gamma={\bf 1}_{f(X)}\}$.
The \emph{fiber of $f$} is defined by
$$
\Sigma f=\frac1{\vol(f(X))}\left\{\int_{f(X)}\gamma(y)dy\ |\ \gamma\in\Gamma_f\right\}\subset V.
$$
\end{definition}

\begin{remark}\label{why-nonsurjective}
(a)  Informally, $\Sigma f$ is the `average fiber' of $f$ over $f(X)$. It is easily observed that $\Sigma f$ is in $\CONV$. If we vary the Borel measures in $\Aff(X)$ and $\Aff(Y)$, the resulting fibers will be mutually isomorphic objects in $\CONV$. Since we are only interested in properties up to isomorphism, whenever we talk on fibers of morphisms in $\CONV$, it is always implicitly assumed that the relevant affine spaces are equipped with arbitrarily chosen translation invariant Borel measures.

\medskip\noindent(b) Definition \ref{fiber} is the straightforward extension of the Billera-Sturmfels \emph{fiber polytope} \cite{Fiber} to the class of compact convex sets. Unlike \cite{Fiber}, we allow non-surjective maps $f$ because, even if $f:X\to Y$ is surjective, the induced map $\Hom(Z,f)$, important in our analysis of $\Sigma$, may fail to be surjective; e.g., for a surjective affine map from a tetrahedron $X$ to a quadrangle $Y=Z$, the identity map $\mathbf 1_Z$ does not lift to $\Hom(Z,X)$.
\end{remark}

\begin{proposition}\label{barycenter}
Let $f:X\to Y$ be in $\CONV$.
\begin{enumerate}[\rm(a)]
\item If $X$ and $Y$ are in $\POL$, then $\Sigma f$ is the following $(\codim f)$-dimensional polytope
\begin{align*}
\Sigma f=\sum_{i=1}^n\frac{\vol(\sigma_i)}{\vol(f(X))}\cdot f^{-1}(x_i),
\end{align*}
where $\{\sigma_1,\ldots,\sigma_n\}$ are the maximal cells of any subdivision of $f(X)$, subdividing the $f$-images of the faces of the polytope $X$, and $x_i\in\sigma_i$ are the barycenters.
\item $\Sigma f\in\CONV$ and $\dim(\Sigma f)=\codim f$.
\end{enumerate}
\end{proposition}

\begin{proof}
(a) This is a slightly extended reformulation of \cite[Theorem 1.5]{Fiber}, which states the equality for the coarsest such subdivision. But exactly the same argument applies to any subdivision. Alternatively, the general case reduces to the case when the coarsest subdivision is the trivial subdivision of $f(X)$, and then the equality follows from the fact that $y\mapsto f^{-1}(y)$ respects affine combinations.

Since $\dim(f^{-1}(x_i))=\codim f$ for all $i$, we have $\dim(\Sigma f)=\codim f$.

\medskip\noindent(b) The argument in the polytopal case \cite[Proposition 1.1]{Fiber}, based on Aumann's 1965 results on integrals of set-valued functions, works here too. Alternatively, the general case can be deduced from the the polytopal case by approximating $X$ from outside by a nested set of polytopes in $\Aff(X)$, containing $X$, and approximating $f(X)$ from outside by the images of these polytopes under the affine extension $\tilde f:\Aff(X)\to\Aff(Y)$. In this case $\Sigma f$ is the intersection of the resulting fiber polytopes.

The dimension equality is a consequence of of the polytopal case in (a). 
\end{proof}

Assume $X$ is a subset of a vector space $V$, in which we have fixed a norm. For a real number $\delta>0$, denote the \emph{$\delta$-neighborhood of $X$} by $U_\delta(X)$; i.e., $U_\delta(X)$ is the union of open $\delta$-discs, centered at the elements of $X$. For two sets $X,Y\subset V$, the \emph{Hausdorff distance} between them is defined by
\begin{align*}
d_{\text H}(X,Y)=\inf\{\delta\ge0\ | X\subset U_\delta(Y)\ \text{and}\ Y\subset U_\delta(X)\}
\end{align*}

The following definition is independent of the choice of a norm in $V$:
\begin{definition}\label{limit}
Let $\epsilon>0$ and $V\in\VECT$. Assume $X_t,X\in\CONV$ and $X_t,X\subset V$ for $0<t<\epsilon$. We write $\lim\limits_{t\to0}X_t=X$ if $\lim\limits_{t\to0}d_{\text H}(X_t,X)=0$.
\end{definition}

The limit set, if it exists, is unique. The uniqueness would fail if we allowed non-closed convex sets. We remark that, despite similar terminology, it is impossible to confuse the two different limits -- functorial and with respect to Hausdorff metric. 

\begin{lemma}\label{continuity}
Let $V,W\in\VECT$, $X_t,Y_t,X,Y\in\CONV$, $X_t,X\subset V$, and $Y_t,Y\subset W$ for $0<t<\epsilon$. Assume $\lim\limits_{t\to0}X_t=X$ and $\lim\limits_{t\to0}Y_t=Y$. We have:
\begin{enumerate}[\rm(a)]
\item $\lim\limits_{t\to0}(X_t+Y_t)=X+Y$, assuming $V=W$;
\item $\lim\limits_{t\to0}\Hom(X_t,Y_t)=\Hom(X,Y)$ in $\Hom(\Aff(X),W)$, assuming $\Aff(X)=\Aff(X_t)$ for all $t$;
\item  If $f:X\to Y$ is an affine map, then $\lim\limits_{t\to0}\Sigma f_t=\Sigma f$, assuming $\Aff(X)=\Aff(X_t)$ for all $t$, where $f_t:X_t\to f_t(X_t)$ is obtained from $f$ by first extending to $\Aff(X)$ and then restricting to $X_t$.
\end{enumerate}
\end{lemma}

\begin{proof} (a) Consider a real number $\delta>0$. For all sufficiently small $t>0$, we have $X\subset U_{\delta/2}(X_t)$ and $Y\subset U_{\delta/2}(Y_t)$, implying $X+Y\subset U_\delta(X_t+Y_t)$. For symmetrical reasons, we also have the inclusions $X_t+Y_t\subset U_\delta(X+Y)$ for all sufficiently small $t$.

\medskip\noindent In (b) the condition on the affine hulls is needed for the existence of an ambient vector space, where the convergence occurs. (We think of $\Hom(\Aff(X),W)$ as $W^{\dim X+1}$.) Without loss of generality we can assume $\Aff(X)=V$ and that there exists an affinely independent set $\{x_0,\ldots,x_d\}\subset X\bigcap\left(\bigcap_{0<t<\epsilon}X_t\right)$, $d=\dim V$. Then there are infinitesimally small perturbations of the images $f(x_i)$ as $t\to0$ such that the perturbed maps $f_t\in\Hom(V,W)$ first bring $f_t(V)$ into $\Aff(Y_t)$ and then ensure the inclusions $f_t(X_t)\subset Y_t$.  This implies the inclusion $\supset$, and the other inclusion is more straightforward. 

\medskip\noindent(c) This is an easy exercise on integrals. 
\end{proof}

\section{No $\Sigma$-covariance}\label{Kernel}

Let $P,Q,R$ and $f:P\to Q$ be in $\CONV$. The convex sets $\Sigma\Hom(R,f)$ and $\Hom(R,\Sigma f)$ are not completely unrelated: (i) if $P,Q,R$ are polytopes than so are these sets, (ii) both have dimension $(\dim R+1)\codim f$ (follows from Proposition \ref{barycenter}(b)), and (iii) $\Sigma\Hom(R,f)\cong\Hom(R,\Sigma f)$ in either of the following three cases: $P=P'\times Q$ and $f$ is the projection map, or $Q$ is a point, or $f$ is injective. But the similarities end here.

The following formula for a centrally symmetric $d$-polytope $S\subset\RR^d$ with respect to $0$ is given in \cite[Corollary 3.6]{Hompolytopes}:
\begin{equation}\label{diamond}
\Hom(S,[0,1])\cong\Diamond(S^{\o})
\end{equation}
where $S^{\o}$ is the \emph{polar} of $S$ and, for any polytope $T\subset\RR^d$ with $0\in\inte(T)$, $\Diamond(T)$ is the \emph{bipyramid} $\conv\big((T,0),(0,1),(0,-1)\big)\subset\RR^{d+1}$.

\begin{theorem}\label{main1}
Let $Q\subset\RR^2$ be a centrally symmetric polygon. Assume $P\subset\RR^3$ is a polytope, such that $\partial Q\times[0,\epsilon]\subset\partial P$ for some $\epsilon>0$, $P$ is combinatorially equivalent to a prism over $Q$, and the opposite facet $Q'\subset P$ is not parallel to $Q$; i.e., $P$ is a slant-truncated right prism over $Q$. Then $\#\vertex\big(\Sigma\Hom(Q,f)\big)\ge\#\vertex\big(\Hom(Q,\Sigma f)\big)+2$ for the orthogonal projection $f:P\to Q$. In particular,  $\Sigma\Hom(Q,f)\not\cong\Hom(Q,\Sigma f)$.
\end{theorem}

\begin{proof}
Without loss of generality we can assume that $0$ is at the center of $Q$. Let $Q$ be an $2n$-gon and $\{v_1,\ldots,v_{2n}\}=\vertex(Q)$, the indexing being cyclic and $\text{mod}(2n)$. The polar $Q^{\o}$ is also a centrally symmetric $2n$-gon. We will identify $Q$ with $(Q,0)$.

Denote $f_*=\Hom(Q,f)$. For any map $\alpha\in\Hom(Q,Q)$, the preimage $f_*^{-1}(\alpha)\subset\Hom(Q,P)$ is a subpolytope, generically of dimension 3. Next we prove the implication
\begin{equation}\label{many-vertices}
\rank\alpha=2\quad\Longrightarrow\quad\#\vertex(f_*^{-1}(\alpha))\ge 2n+4.
\end{equation}

Assume $\dim(\Im\alpha)=2$. Then the subpolygon $Q_\alpha:=\alpha(Q)\subset Q$ is isomoprhic to $Q$. Let $P_\alpha$ be the maximal truncated right prism inside $P$ with $Q_\alpha$ as the base. Denote by $Q_\alpha'$ the facet of $P_\alpha$, opposite to $Q_\alpha$. Let $w_i=\alpha(v_i)$ and $w_i'$ be the corresponding vertices of $Q_\alpha'$.

The elements of $f_*^{-1}(\alpha)$ can be interpreted as the affine planes in $H\subset\RR^3$, meeting all vertical edges of $P_\alpha$: if $\{x_i\}=(\alpha(v_i)\times\RR_{\ge0})\cap H$, then the map corresponding to $H$ is defined by $v_i\mapsto x_i$, $i=1,\ldots,2n$. After this interpretation, the vertices  of $f_*^{-1}(\alpha)$ correspond to the planes $H$ which do not fit in a \emph{smooth 1-family} of affine planes, satisfying the same condition. Here, under a `smooth 1-family' we mean a system $\{H_t\}_{(-1,1)}$ of affine planes in $\RR^3$ such that the intersection point $H_t\cap\big(\RR_{\ge0}(0,0,1)\big)$ and  the unit normals to $H_t$ are both smooth functions of $t$, and we say that $H$ `fits' in such a system if $H=H_0$. This \emph{smooth perturbation criterion} for the vertices of a polytope is crucial  in \cite{Hompolytopes,Vertices} for studying the vertex sets of various hom-polytopes. The planes $H$, corresponding to the vertices of $f_*^{-1}(\alpha)$, will be called \emph{tight}. 

For every index $i$, we can rotate the coordinate plane $(\RR^2,0)$ in $\RR^3$ about the axis $\Aff(w_i,w_{i+1})$, staying within the family of planes corresponding to $f_*^{-1}(\alpha)$, until we hit the polygon $Q_\alpha'$. Let $H_i$ be the corresponding extremal position of the rotated plane. Then $H_i$ is tight, representing a vertex $z_i\in f_*^{-1}(\alpha)$. Similarly, every edge $[w'_i,w'_{i+1}]\subset Q_\alpha'$ gives rise to a vertex of $z_i'=f_*^{-1}(\alpha)$. We have $z_i\not=z_j$ and $z'_i\not=z'_j$ for $i\not=j$, and $z_i=z'_j$ if the plane $H_i=H'_j$ contains the corresponding edges of $Q_\alpha$ and $Q'_\alpha$. In particular, if there is an index $i$, such that $H_i\cap Q_\alpha'\in\vertex(Q_\alpha')$, then
\begin{equation}\label{2n+1}
\#\{z_1,\ldots,z_{2n},z'_1,\ldots,z'_{2n}\}\ge 2n+1.
\end{equation}

The existence of such an index follows from the condition that the planes $(\RR^2,0)$ and $\Aff(Q')$ are not parallel. In fact, let $w_k'$ be on the minimal height among the vertices of $Q_\alpha'$, as measured by the third coordinate. There can be at most one more vertex of $Q_\alpha'$ on the same height, and if such exists it must be adjacent to $w_k'$. We can assume that $w_{k+1}'$ is strictly higher than $w_k'$. Consider the plane $H_{n+k}$ through the edge of $Q_\alpha$, opposite to $[w_k,w_{k+1}]$; see Figure 1. 

\begin{figure}[h!]
\caption{The plane $H_{n+k}$}
\vspace{.15in}
\includegraphics[
scale=3.5]{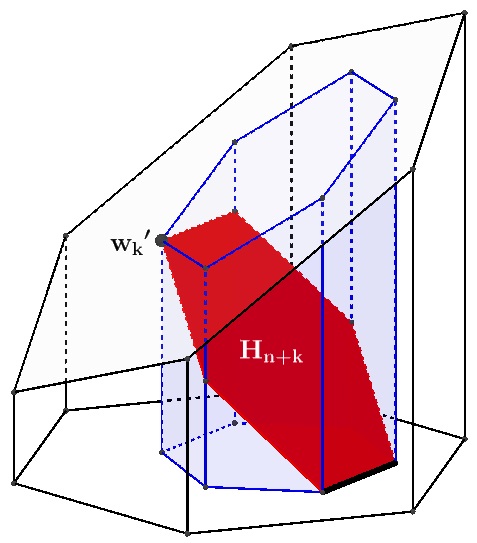}
\end{figure}

The height function on the $2n$-gon $H_{n+k}\cap\big(Q\times\RR_{\ge0}\big)$ is maximized along the segment $H_{n+k}\cap\big([w_k,w_{k+1}]\times\RR_{\ge0}\big)$. In paricular, $H_{n+k}$ has the desired property: $H_{n+k}\cap Q_\alpha'=\{w_k'\}$.

Next we strengthen (\ref{2n+1}) to the inequality
$\#\{z_1,\ldots,z_{2n},z'_1,\ldots,z'_{2n}\}\ge 2n+2$ by observing that there is always a second pair of vertices $(w',w'')$ of $Q'_\alpha$ with $w'$ on the minimal height among the vertices of $Q'_\alpha$ and $w''$ adjacent and strictly higher than $w'$. (If $Q'_\alpha$ has two vertices on the minimal height then $w'\not=w'_k$.)

We also have the two vertices of $f_*^{-1}(\alpha)$, corresponding to the planes $(\RR^2,0)$ and $\Aff(Q')$. Since they do not belong to  $\{z_1,\ldots,z_{2n},z'_1,\ldots,z'_{2n}\}$, we derive (\ref{many-vertices}). 

For a generic element $\alpha\in\Im(f_*)$, we have $\dim(\Im\alpha)=2$. Therefore, the inequality (\ref{many-vertices}), Lemma \ref{basic-sum}(b), and Proposition \ref{barycenter}(a) imply $\#\vertex\big(\Sigma\Hom(Q,f)\big)\ge2n+4$.

On  the other hand, since $\Sigma f\cong[0,1]$, (\ref{diamond}) implies $\Hom(Q,\Sigma f)\cong\Diamond(Q^{\o})$, and this in turn implies
$\#\vertex(\Hom(Q,\Sigma f))=2n+2$.
\end{proof}

\section{Minkowski sum covariance}\label{Sum-covariance}

Before developing a correct version of the fiber equality in Section \ref{Conker}, we investigate the functorial behavior of the Minkowski sum. We will need the following well known fact (e.g.,  \cite[Proposition 2.1]{Hompolytopes}):
\begin{lemma}\label{folklore}
For $P,Q\in\POL$, the polytope $\Hom(P,Q)$ has dimension $(\dim P+1)\dim Q$ and its facets are the subsets 
$$
H(v,F):=\{\phi\in\Hom(P,Q)\ |\ \phi(v)\in F\}\subset\Hom(P,Q),
$$
where $v\in\vertex(P)$ and $F\in\FF(Q)$.
\end{lemma}

\begin{lemma}\label{minkowskisum}
Let $Q$ and $R$ be polytopes in a vector space $V$. Then, for any polytope  $P$, we have $\Hom(P,Q+R)=\Hom(P,Q)+\Hom(P,R)$.
\end{lemma}
\noindent (The Minkowski sum of hom-sets is taken in the vector space of affine maps $P\to V$.) 

\begin{proof}
First we reduce the general case to the case when $Q$ and $R$ are full-dimensional and $\cN(Q)=\cN(R)$. 

Without loss of generality, $0\in Q\cap R$ and $V=\RR Q+\RR R$. For a real number $t>0$, consider the polytopes $Q_t=Q+t R$ and $R_t=R+t Q$. We have:
\begin{enumerate}[\rm$\centerdot$]
\item $\dim(Q_t)=\dim(R_t)=\dim V$,
\item $\cN(Q_t)=\cN(R_t)$ (Lemma \ref{basic-sum}(a)),
\item $\lim\limits_{t\to0}Q_t=Q$ and $\lim\limits_{t\to0}R_t=R$.
\end{enumerate}

By Lemma \ref{continuity}(a,b), it is enough to prove Lemma \ref{minkowskisum} for $Q_t$ and $R_t$ with $t>0$ sufficiently small. This way we have reduced the general case to full-dimensional polytopes with equal normal fans.

By Lemma \ref{basic-sum}(a), $\cN(Q)=\cN(R)=\cN(Q+R)$. By Lemma \ref{folklore}, this equality implies that, for a vertex $x\in P$ and a pair of corresponding facets $F\subset Q$ and $G\subset R$, the three facets $H(x,F)\subset\Hom(P,Q)$, $H(x,G)\subset\Hom(P,R)$, and $H(x,F+G)\subset\Hom(P,Q+R)$ are parallel, i.e., represent the same 1-cone in the common normal fan (notation as in Lemma \ref{folklore}). Then, by Lemmas \ref{folklore} and \ref{basic-sum}(a), we have
\begin{align*}
\cN(\Hom(P,Q))=\cN(\Hom(P,R))&=\cN(\Hom(P,Q+R))\\
&=\cN(\Hom(P,Q)+\Hom(P,R)).
\end{align*}
Consequently, it is enough to show that the interiors of corresponding pairs of facets of $\Hom(P,Q+R)$ and $\Hom(P,Q)+\Hom(P,R)$ meet.

Lemma \ref{folklore} implies that the interior points of the facets $H(x,F)\subset\Hom(P,Q)$ and $H(x,G)\subset\Hom(P,R)$ are, respectively, the sets
\begin{align*}
&\big\{f\in H(x,F)\ |\ f(x)\in\inte(F),\ f\big(\vertex(P)\setminus\{x\}\big)\subset\inte(Q)\big\},\\
&\big\{g\in H(x,G)\ |\ g(x)\in\inte(G),\ g\big(\vertex(P)\setminus\{x\}\big)\subset\inte(R)\big\}.
\end{align*}
By Lemma \ref{basic-sum}(c), for such $f$ and $g$, the sum $f+g$ is in the interior of the corresponding facet of $\Hom(P,Q)+\Hom(P,R)$. But it is also in the interior of the facet $H(x,F+G)\subset\Hom(P,Q+R)$ by the similar description of the latter.
\end{proof}

\begin{remark}\label{why-limits}
In the proof of Lemma \ref{minkowskisum}, the initial reduction to polytopes with equal normal fans seems unavoidable. The reason for this is the lack of control of the normal fan of $\Hom(P,Q)+\Hom(P,R)$ for general $Q$ and $R$. It is this reduction step where limit sets enter the picture, even if one wants to prove Lemma \ref{minkowskisum} for full-dimensional polytopes. On the other hand, a convex set is the same as a filtered union of polytopes. This, together with the uniqueness of limits, explains why $\CONV$ is the optimal framework for our functorial approach.
\end{remark}

\begin{example}\label{no-cintravariance} 
The contravariant version of Lemma \ref{minkowskisum} is false; i.e., in general, $\Hom(P,R)+\Hom(Q,R)\not\cong\Hom(P+Q,R)$. Consider the following rectangles: $P=\conv\big((2,1),(2,-1),(-2,1),(-2,-1)\big)$, $Q=\conv\big((1,2),(1,-2),(-1,2),(-1,-2)\big)$, $R=[0,1]$.
Since $P+Q=[-3,3]^2$, by (\ref{diamond}) we have $\Hom(P+Q,R)\cong\Diamond\big([-1,1]^2\big)$. On the other hand, the polar polytopes $P^{\o}$ and $Q^{\o}$ are central parallelograms in $\RR^2$, related by a $90^{\o}$-rotation and with unequal diagonals along the coordinate axes. In particular, $P^{\o}+Q^{\o}$ is a central octagon in $\RR^2$. Together with (\ref{diamond}), this implies that the bipyramids  $\Hom(P,R)\cong\Diamond(P^{\o})$ and $\Hom(Q,R)=\Diamond(Q^{\o})$ are related by a $90^{\o}$-rotation around the axis $\RR(0,0,1)$, implying in turn $\Hom(P,R)+\Hom(Q,R)\cong\Diamond(P^{\o}+Q^{\o})$. But the polytopes $\Diamond(P^{\o}+Q^{\o})$ and $\Diamond\big([-1,1]^2\big)$ have different numbers of facets: 16 vs. 8.
\end{example}

\begin{corollary}\label{linear}
\begin{enumerate}[\rm(a)]
\item
Let $Q_1,\ldots,Q_m\in\CONV$ be in a same vector space. Then
$\Hom(-,Q_1+\cdots+Q_m)\cong\Hom(-,Q_1)+\cdots+\Hom(-,Q_m)$ in $\CONV^{\CONV^{\op}}$, or in $\POL^{\POL^{\op}}$ if $Q_1,\ldots,Q_m\in\POL$.
\item
Let $\epsilon,\lambda_1,\ldots,\lambda_k>0$, $0<t<\epsilon$, and $i=1,\ldots,k$. Assume $X,X_t,Y_i,Y_{it}\in\CONV$, $\Aff(X)=\Aff(X_t)$, and $Y_i,Y_{it}\subset V\in\VECT$. Assume $\lim\limits_{t\to0}X_t=X$ and $\lim\limits_{t\to0}Y_{it}=Y_i$. Then 
$\lim\limits_{t\to0}
\Hom(X_t,\lambda_1Y_{t1}+\cdots+\lambda_kY_{tk})=
\lambda_1\Hom(X,Y_1)+\cdots+\lambda_k\Hom(X,Y_k)$  in $\Hom(\Aff(X),V)$.
\end{enumerate}
\end{corollary}

\begin{proof}
(a) By Lemma \ref{minkowskisum}, the tautological embedding
$$
\xymatrix{
\Hom(P,Q_1)+\cdots+\Hom(P,Q_m)\ar@{^{(}->}[r]&\Hom(P,Q_1+\cdots+Q_m)
}
$$
is surjective. But it is also natural in $P$.

\medskip\noindent (b) In view of Lemmas \ref{continuity}(a,b) and \ref{minkowskisum}, one only needs to represent the convex sets as limits of polytopes and use $\Hom(X,\lambda Y)=\lambda\Hom(X,Y)$.
\end{proof}

\section{Affine-compact kernel}\label{Conker} 

In analogy with the functor $\Hom_\VECT(-,\ker f):\VECT\to\VECT$, for a map $f:X\to Y$ in $\CONV$, we introduce the following contravariant functor:
\begin{align*}
&\ker^*(f):\CONV\to\SETS,\\
&\ker^*(f)(Z)=\{g:Z\to X\ |\ g\ \text{affine and }\ f(g(Z))\ \text{a singleton}\},\\
&\ker^*(f)(h):g\mapsto gh\ \ \text{for}\ \ h:Z'\to Z\ \ \text{in}\ \ \CONV.
\end{align*}

An alternative definition is provided by the following \emph{pull-back diagram} in $\SETS$, natural in $Z$, which also introduces the important map $\Hom(Z,f)^{\ev}$:
\begin{equation}\label{pull-back}
\xymatrix{
\ker^*(f)(Z)\ar@{^{(}->}[r]\ar[dd]_{\Hom(Z,f)^{\ev}}&\Hom(Z,X)\ar[dd]^{\Hom(Z,f)}\\
\\
Y\ar[r]_{\text{const.}}&\Hom(Z,Y)
}
\end{equation}
where (i)  to every point $y\in Y$ the bottom map assigns the constant map $Z\to Y$ with value $y$, and (ii) $\Hom(Z,f)^{\ev}$ is the evaluation map $g\mapsto(fg)(Z)$.

\begin{proposition}\label{Functor}
In the notation introduced above,
\begin{enumerate}[\rm(a)]
\item
$\ker^*(f)$ is in $\CONV^{\CONV^{\op}}$, or
in $\POL^{\POL^{\op}}$ if $f$ is in $\POL$.
\item $\ker^*(f)$ is not a representable functor, unless $f(X)$ is a singleton.
\item In $\CONV^{\CONV^{\op}}$, or in $\POL^{\POL^{\op}}$ if $f$ is in $\POL$, we have
$$
\ker^*(f)=\lim\limits_{\longrightarrow}\big(\Hom\big(-,f^{-1}(y)\big)\ |\ y\in Y\big).
$$ 
\item $\dim(\ker^*(f)(Z))=(\dim Z+1)\codim f+\dim(f(X))$,
\item $\Im\big(\Hom(Z,f)^{\ev}\big)=f(X)$.
\end{enumerate}
\end{proposition}

\begin{proof} (a) Since the assignment $h\mapsto gh$ is an affine function of $g$, we only need to show $\ker^*(f)(Z)\in\CONV$, with  $\ker^*(f)(Z)\in\POL$ for $X,Y,Z\in\POL$. But because the limits in $\SETS$, $\CONV$, $\POL$ agree (Section \ref{preliminaries}), this claim follows from the pull-back diagram (\ref{pull-back}), where the right and bottom arrows are affine maps.

\medskip\noindent(b) If there is an `affine-compact kernel' object $\ker(f)\in\CONV$, such that $\ker^*(f)=\Hom(-,\ker(f))$, then we have $\dim(\ker^*f(Z))\big)=(\dim Z+1)\dim(\ker f)$ (Lemma \ref{folklore}). But by (d), $\dim(\ker^* f(Z))=(\dim Z+1)\codim f+\dim(f(X))$ for \emph{every} $Z$. This is a contradiction, unless $\dim(f(X))=0$.

\medskip\noindent(c) The colimit equality is straightforward. Alternatively, the $\CONV$- and $\POL$-contravariant versions of the standard colimit representations, given in \cite[Ch.3, \S7]{Categories}, produces a much larger non-discrete category on which the corresponding universal co-cone is based. However, the colimits over the connected components of that category are exactly the representable functors on the right hand side of (c).

\medskip\noindent(d) This follows from the dimension formula in Lemma \ref{folklore} and the equality 
$$
\dim(f^{-1}(y))=\codim f
$$
for generic $y\in f(X)$.

\medskip\noindent(e) The inclusion $\Im\big(\Hom(Z,f)^{\ev}\big)\subset f(X)$ is obvious and the opposite inclusion follows by considering the constant maps $Z\to X$. 
\end{proof}

Next, for a map $f:X\to Y$ in $\CONV$, we introduce the following functor 
$$
\Sigma\Hom(-,f)^{\ev}\in\CONV^{\CONV^{\op}}.
$$
To $Z\in\CONV$ it assigns $\Sigma\Hom(Z,f)^{\ev}$; notation as in the diagram (\ref{pull-back}). For an affine map $h:Z'\to Z$, we first work out the polytopal case $X,Y,Z\in\POL$. By Proposition \ref{Functor}(e), $\Im\big(\Hom(Z,f)^{\ev}\big)=\Im\big(\Hom(Z',f)^{\ev}\big)=f(X)$. Let $\sigma_1,\ldots,\sigma_n$ be the maximal cells of a subdivision of $f(X)$, which subdivides the images of faces of both polytopes $\Hom(Z,\ker^*(f))$ and $\Hom(Z',\ker^*(f))$. Then, using Proposition \ref{barycenter}(a), together with the notation introduced there, we can define the map $\Sigma\Hom(h,f)^{\ev}:\Sigma\Hom(Z,f)^{\ev}\to\Sigma\Hom(Z',f)^{\ev}$ as follows:
$$
\xymatrix{
\sum_{j=1}^n\frac{\vol(\sigma_j)}{\vol(f(X))}\cdot\big(\Hom(Z,f)^{\ev}\big)^{-1}(x_j)\ar@{=}[r]\ar@{=}[d]&
\Sigma\Hom(Z,f)^{\ev}\\
\sum_{j=1}^n\frac{\vol(\sigma_j)}{\vol(f(X))}\cdot \Hom(Z,f^{-1}(x_j))\ar[r]^{-\ \circ\ h}&
\sum_{j=1}^n\frac{\vol(\sigma_j)}{\vol(f(X))}\cdot \Hom(Z',f^{-1}(x_j))\ar@{=}[d]\\
\Sigma\Hom(Z',f)^{\ev}&\sum_{j=1}^n\frac{\vol(\sigma_j)}{\vol(f(X))}\cdot\big(\Hom(Z',f)^{\ev}\big)^{-1}(x_j)\ar@{=}[l]
}
$$
Checking that we get a functor in $\POL^{\POL^{op}}$ is straightforward, with a similar use of Proposition \ref{barycenter}(a).

Now assume $Z$  and $f:X\to Y$ are in $\CONV$, $Z=\lim\limits_{t\to0}Z_t$, and $X=\lim\limits_{t\to0}X_t$, where $Z_t\subset Z$ and $X_t\subset X$ are in $\POL$, satisfying the condition $\dim(Z_t)=Z$ and $\dim(X_t)=\dim X$ for all $0<t<\epsilon$. Put $f_t=f|_{X_t}:X_t\to f(X_t)$. Lemma \ref{continuity}(b) and the pull-back diagrams for $Z_t$ and $f_t$, similar to (\ref{pull-back}), imply the convergence $\lim\limits_{t\to0}\Hom(Z_t,\ker^*(f_t))=\Hom(Z,\ker^*(f))$ in the ambient space $\Hom(\Aff(Z),\Aff(X))$. Then Lemma \ref{continuity}(c) implies
\begin{equation}\label{limit-equality}
\lim\limits_{t\to0}\Sigma\Hom(Z_t,f_t)^{\ev}=\Sigma\Hom(Z,f)^{ev}.
\end{equation}

Let $h:Z'\to Z$ be in $\CONV$ and $Z'=\lim\limits_{t\to0}Z_t'$ with $\dim(Z_t')=\dim Z'$ for all $t$. We can additionally assume $h(Z_t')\subset Z_t$ for all $t$. Denote $h_t:=h|_{Z_t}:Z'_t\to Z_t$. The definition of our functor in the polytopal case above ensures the compatibility $\Hom(h_t,f_t)^{\ev}=\Hom(h_s,f_t)^{\ev}$ on $\Hom(Z_t,f)^{\ev}\cap\Hom(Z_s,f)^{\ev}$. In particular, the limit equality (\ref{limit-equality}) gives rise to a functorial map $\Sigma\Hom(Z,f)^{ev}\to\Sigma\Hom(Z',f)^{\ev}$.

We are ready to state an affine substitute for the failed $\Sigma$-covariance.

\begin{theorem}\label{kaf=fiber}
If $f$ is in $\CONV$, then $\Sigma\Hom(-,f)^{ev}\cong\Hom(-,\Sigma f)$ in $\CONV^{\CONV^{\op}}$, or in $\POL^{\POL^{\op}}$ if $f$  is in $\POL$. 
\end{theorem}

\begin{proof} 
First we consider the case, when $Z$ and $f:X\to Y$ are in $\POL$. By Proposition \ref{Functor}(e), $\Im(\Hom(Z,f)^{\ev})=f(X)$. Let $\sigma_1,\ldots,\sigma_n\subset Y$ be the maximal cells of a subdivision of $f(X)$, which subdivides the images of faces of $\Hom(Z,\ker^*(f))$ as well as the images of faces of $X$. We have
{\small
\begin{align*}
&\Sigma\Hom(Z,f)^{\ev}=\sum_{j=1}^m\frac{\vol(\sigma_j)}{\vol(f(X))}\cdot \Hom(Z,f^{\ev})^{-1}(x_j)=\\
&\sum_{j=1}^m\frac{\vol(\sigma_j)}{\vol(f(X))}\cdot \Hom(Z,f^{-1}(x_j))\cong
\Hom\left(Z,\sum_{j=1}^m\frac{\vol(\sigma_j)}{\vol(f(X))}\cdot f^{-1}(x_j)\right)=\Hom(Z,\Sigma f),
\end{align*}
}

\noindent where the first and last equalities follow from Proposition \ref{barycenter}(a) (notation as in that proposition) and the middle isomorphism is provided by Corollary \ref{linear}(a).

The general case, when $Z$ and $f$ are in $\CONV$, can be derived from the polytopal case along the lines the functor $\Sigma\Hom(-,f)^{\ev}\in\CONV^{\CONV^{\op}}$ was constructed in two steps, first considering the polytopal case.
\end{proof}

\section{Affine-compact Cokernel}\label{Cokernel} 

The following definition is modeled after the cokernel isomorphisms in $\VECT$, mentioned in Section \ref{Yoneda_embedding}.

\begin{definition}\label{two-cokernels}
For a map $f:X\to Y$ in $\CONV$, we have the \emph{object}
$$
\coker(f)=\lim\limits_{\longrightarrow}
\big(\xymatrix{X\ar@/^/[r]^f \ar@/_/[r]_{\text{const.}}&Y
}\big)\in\CONV,
$$
where the lower arrow in the diagram is a constant map with the value a point in $f(X)$, and the \emph{functor} $\coker^*(f)\in\SETS^{\CONV^{op}}$, defined by
\begin{align*}
&\coker^*(f)(Z)=\Hom(Z,Y)/(g_1\sim g_2\ \text{iff}\ \rho g_1=\rho g_2\ \text{for any}\ \rho:Y\to Y'\\ 
&\ \ \qquad\qquad\qquad\qquad\qquad\qquad\qquad\qquad\text{in}\ \CONV\ \text{with}\ (\rho f)(X)\ \text{a singleton}),\\
&\coker^*(f)(h):[g]\mapsto [gh]\quad\text{for}\ h:Z'\to Z\ \text{and}\ g:Z\to Y\ \text{in}\ \CONV.
\end{align*}
\end{definition}

First, we observe that $\coker(f)$ is independent of the target of the constant map, evaluating in $f(X)$, and it can be identified with $\pi(Y)$, where $\pi:\Aff(Y)\to\Aff(Y)$ is an affine map with $\pi^{-1}(\pi(f(x)))=\Aff(f(X))$ for any $x\in X$; i.e., $\pi$ is a linear projection of $Y$ along $\Aff(X)$. The covariant representable functor $\Hom(\coker(f),-)\in\CONV^{\CONV}$ identifies as follows:
\begin{align*}
&\Hom(\coker(f),Z)=\{g:Y\to Z\ |\ g\ \text{affine and }\ g(f(X))\ \text{a singleton}\},\\
&\Hom(\coker(f),h)(g)=hg\quad\text{for}\ g:Y\to Z\ \text{and}\ h:Z\to Z'\ \text{in}\ \CONV.
\end{align*}

As for the functor $\coker^*(f)$, it can be put in a more general framework. Observe that any map $\pi:S\to T$ in $\CONV$ gives rise to the functor:
\begin{align*}
&\Im\big(\Hom(-,\pi)\big)\in\CONV^{\CONV^{\op}},\\
&\Im\big(\Hom(Z,\pi)\big)=\Im\big(
\xymatrix{\Hom(Z,S)\ar[rr]^{\Hom(Z,\pi)}&&\Hom(Z,T)}
\big),\\
&\Im(\Hom(h,\pi))(g)=gh\quad\text{for}\ h:Z'\to Z\ \text{in}\ \CONV\ \text{and}\ g\in\Im(\Hom(Z,\pi)).
\end{align*}
For $f$ as in Definition \ref{two-cokernels}, let $\pi:Y\to\coker(f)$ be the canonical map. Then, because $[g_1]=[g_2]$ in $\coker^*(f)(Z)$ if and only if $\pi g_1=\pi g_2$, inside $\CONV^{\CONV^{\op}}$ we have:
$$
\coker^*(f)\cong\Im\big(\Hom(-,\pi)\big).
$$

Using barycentric coordinates and the alternative description of $\coker^*(f)$ above, one easily derives
\begin{lemma}
The functor
$\coker^*(f)$ is in $\CONV^{\CONV^{\op}}$, or in $\POL^{\POL^{\op}}$ if $f$ is in $\POL$.
\end{lemma}

The next proposition clarifies the relationship between the two cokernels. 

\begin{proposition}\label{coker-relationship}
Let $f:X\to Y$ be in $\CONV$ and $\pi:Y\to\coker(f)$ be the canonical map.
The functor $\coker^*(f)$ is representable if and only if $\pi$ has an affine section, in which case $\coker^*(f)=\Hom(-,\coker(f))$.
\end{proposition}

\begin{proof} If $\pi$ has a section $\sigma:\pi(Y)\to Y$, then $\Hom(-,\sigma)$ is a right inverse of $\Hom(-,\pi)$. In particular, $\Hom(Z,\pi)$ is surjective for every $Z$ or, equivalently, $\coker^*(f)\cong\Hom(-,\pi(Y))$.

Conversely, if $\coker^*(f)$ is representable, then, by applying it to a singleton, we get $\Hom(-,\pi(Y))\cong\coker^*(f)$ and, in particular, an automorphism $\tau:\pi(Y)=\Hom(\star,\pi(Y))\to\coker^*(f)(\star)=\pi(Y)$. The isomorphism $\Hom(\pi(Y),\pi(Y))\cong\coker^*(f)(\pi(Y))$ maps $1_{\pi(Y)}$ to $\tau$. Thus, $\tau$ lifts to $Y$. But then $1_{\pi(Y)}=\tau^{-1}\tau$ also lifts to $Y$ or, equivalently, $\pi$ has an affine section.  
\end{proof}

\begin{remark}\label{second-failure}
The relationship between the cokernel and the fiber constructions is not straightforward: usually $\Hom(\coker(f),Z)\not\cong\Sigma\Hom(f,Z)$. For example, if $X=Y$ and $f={\bf1}_X$, then $\Hom(\coker(f),Z)\cong Z$ and $\Sigma\Hom(f,Z)$ is a point.
\end{remark}

\section{Sandwiching and complementing}\label{Difference} 

\subsection{Sandwiching} Let $U,V,W$ be vector spaces, with $U\subset V$. We want to describe $\Hom(W,V/U)$ without referring to quotient linear structures. A possible solution is provided by the subset $W_{U,V}=\{f\in\Hom(W,V)\ |\ U\subset f(W)\}$. In fact, the composite map 
$W_{U,V}\hookrightarrow\Hom(W,V)\to\Hom(W,V/U)$ is surjective and $W_{U,V}$ is the smallest of such choices inside $\Hom(W,V)$, retaining some linear structure: it is invariant under scaling by non-zero real numbers. Informally, $W_{U,V}$ corresponds to making $U$ a necessary target for linear maps $W\to V$, like $0$ is the necessary target for any linear map. 

In the polytopal setting, a similar object is studied in \cite{Sandwich}. Namely, for any two $d$-polytopes $Q\subset R$, the \emph{space of sandwiched simplices} is the subspace
$$
\Delta_{Q,P}=\{\Delta\ |\ Q\subset\Delta\subset P,\ \Delta\ \text{a $d$-simplex}\}\subset\RR^{d(d+1)},
$$
where the embedding into the Euclidean space results from a particular enumeration of the vertices of $\Delta$.  This is a complicated semialgebraic set and \cite{Sandwich} employs Morse theory to analyze it.

We can extend the construction to $\CONV$ as follows. Let $X,Y,Z$ be in $\CONV$, with $Y\subset X$. Consider the subset
$$
Z_{Y,X}=\{g:Z\to X\ |\ Y\subset g(Z)\}\subset\Hom(Z,X).
$$
The set $Z_{Y,X}$ is not a functorial construction in the following sense: for two inclusions $Y\subset X$, $Y'\subset X'$ and two maps $\psi:Z'\to Z$, $\theta:X\to X'$ in $\CONV$, such that $\theta(Y)\subset Y'$, the assignment $g\mapsto\theta g\psi$ does not always define a map $Z'_{Y,X}\to Z_{X',Y'}$. The reason is that the image of $\theta g\psi$ may easily fail to contain $Y'$, even if $\psi$ is the identity map ($\theta(Z')$ can be too small) or $\theta$ is the identity map ($Y'$ can be too large).

In order to make the $Z_{Y,X}$ into a functorial construction, we invoke categories of factorizations in the sense of \cite{Nostalgia}. For a category $\cC$, the objects of its \emph{category of factorizations} $\F\cC$ are morphisms in $\cC$, and a morphism from $f':a'\to b'$ to $f:a\to b$ is a commutative square in $\cC$ 
\begin{align*}
\xymatrix{
a\ar[r]^f&b\ar[d]^\theta\\
a'\ar[r]_{f'}\ar[u]^\phi&b'\\
}
\end{align*}
The composition in $\F\cC$ is by concatenating two squares and taking the composition along the vertical edges. This is different from the \emph{category of arrows}  \cite[Ch.2]{Categories}, where the morphisms are commutative squares with similar horizontal arrows but whose vertical arrows are oriented upward. The category of arrows for $\CONV$ contains the category of pairs $Y\subset X$, where the morphisms are affine maps $X\to X'$, mapping $Y$ to $Y'$. But we have already observed that $Z_{Y,X}$ is not functorial with respect to such maps of pairs. 

One more notation: $\CONV_{\surj}$ is the subcategory of $\CONV$ with the same objects and surjective affine maps.

We introduce the following \emph{sandwiching} functor

\begin{align*}
&\circledS:\CONV_{\surj}^{\op}\times\FCONV^{\op}\to\COMP,\\
&\circledS(Z,f)=\{g\in\Hom(Z,X)\ |\ f(Y)\subset g(Z)\},\\
&\circledS
\begin{pmatrix}
\rho,&
\xymatrix{
Y\ar[r]^f&\ar[d]^\theta X\\
Y'\ar[r]_{f'}\ar[u]^\phi&X'
}
\end{pmatrix}
(g)=\theta g\rho\quad\text{for}\quad\rho:Z'\to Z\ \text{in}\ \CONV_{\surj}.
\end{align*}

Observe that the functor $\circledS$ is (i) affine-compact covariant in $X$, (ii) affine-compact contravariant in $Y$, and (iii) affine-compact contravariant in $Z$ with respect to surjective maps.

\begin{proposition}\label{about-sanwishing}
Assume $X,Y,Z$ and $f:Y\to X$ are in $\POL$. If $\dim(\Im f)=\dim X$, then $\circledS(Z,f)$ is a semialgebraic set.
\end{proposition}

\begin{proof}
Let $d=\dim X$. The condition $\circledS(Z,f)\not=\emptyset$ implies that $\dim Z\ge d$. For $g\in\circledS(Z,f)$, the inclusion $f(Y)\subset g(Z)$ is equivalent to the condition that the $(d-1)$-dimensional images of facets of $Z$ have the vertices of $f(Y)$ on their non-negative sides. This leads to several systems of determinantal inequalities in the coefficients of the (non-homogeneous) linear forms defining $g$. The systems depend on the families of facets of $Z$, which have  $(d-1)$-dimensional images under $g$. One considers all possible $2^{\#\FF(Z)}$ families, some geometrically not feasable, i.e., having the empty solution set. The condition that a particular facet of $Z$ has a $(d-1)$-dimensional image under $g$ also expresses as (a pair of) determinantal inequalities in the entries of $g$.
\end{proof}

Below, in Remark \ref{nonlinearity}, we give an example of $\circledS(Z,f)$, which can not be described by a Boolean combination of linear inequalities. We also remark that, using a more elaborate argument, one can actually drop the condition  $\dim(\Im f)=\dim X$ in the proposition above.

\subsection{Complementing} There is an associated and, in a sense, complementary assignment, more amenable to topological control.

Let $\F\CONV_{\inj}$ denote the category with the same objects as $\F\CONV$, where the morphisms are the commutative squares
$$
\xymatrix{
Y\ar[r]^f&\ar[d]^\theta X\\
Y'\ar[r]_{f'}\ar[u]^\phi&X'
}
$$
with $\theta$ injective.

We introduce the following \emph{complementing} functor:
\begin{align*}
&\copyright:\CONV^{\op}\times\F\CONV^{\op}_{\inj}\to\COMP,\\
&\copyright(Z,f)=\overline{\{g\in\Hom(Z,X)\ |\ g(Z)\cap\Im f=\emptyset\}},\\
&\copyright
\begin{pmatrix}
\rho,&
\xymatrix{
Y\ar[r]^f&\ar[d]^\theta X\\
Y'\ar[r]_{f'}\ar[u]^\phi&X'
}
\end{pmatrix}
(g)=\theta g\rho.
\end{align*}
The overline in the definition of $\copyright(Z,f)$ refers to the closure in the Euclidean topology. We have to take the closure to get an affine compact functor. In fact, $\copyright$ is (i) affine-compact contravariant in $Z$, (ii) affine-compact contravariant in $Y$, and (iii) affine-compact covariant in $X$ with respect to injective maps.

The link to cokernel objects is that the functor $\copyright$ makes $\inte(f(Y))$ an impossible target when $\dim(f(Y))=\dim X$, a fusion of the topological quotient and affine maps: for topological spaces $Y\subset X$, the space of continuous maps $Z\to X\setminus Y$ in open-compact topology is homeomorphic to that of the continuous maps $Z\to\big((X/Y)\setminus\{\star\}\big)$, where $\star\in X/Y$ is the point corresponding to the subspace $Y\subset X$. 

For a map $f:Y\to X$ in $\CONV$, let $\mathcal P(X\setminus f)$ be the set of convex compact subsets $X'\subset X$, admitting affine functions $\alpha:X'\to\RR$ for which $\alpha(X')\subset\RR_{\ge0}$ and $\alpha(\Im f)\subset\RR_{\le0}$. The inclusion order on $\mathcal P(X\setminus f)$ makes it into a poset and, therefore, a subcategory of $\CONV$. 

Denote by $\Hom(-,\mathcal P(X\setminus f))$ the image of the embedding
$\mathcal P(X\setminus f)\to\CONV^{\CONV^{\op}}$,
which is the restriction of the Yoneda Embedding $\CONV\to\CONV^{\CONV^{\op}}$; see Section \ref{preliminaries}. In particular, we have a natural transformation 
$$
\Hom(-,X')\xymatrix{\ar[r]^{\bullet}&}\Hom(-,X'')\\ 
$$
whenever $X'\subset X''$ in $\mathcal P(X\setminus f)$.

\begin{proposition}\label{about-complementing1} Let $X,Y,Z$ and $f:Y\to X$ be in $\CONV$.
\begin{enumerate}[\rm(a)]
\item The following equality holds in $\COMP^{\CONV^{\op}}$:
$$
\copyright(-,f)=\lim\limits_{\longrightarrow}\Hom(-,\mathcal P(X\setminus f));
$$
\item 
We have 
$$
\quad \ \ \copyright(Z,f)=
\begin{cases}
\overline{{\{g\in\Hom(Z,X)\ |\ g(Z)\cap\inte(\Im f)=\emptyset\}}},\\
\{g\in\Hom(Z,X)\ |\ g(Z)\cap\inte(\Im f)=\emptyset\}\ \text{if}\ \dim(\Im f)=\dim X;
\end{cases}
$$
\item $\copyright(Z,f)$ is not a convex set unless $\dim(\Im f)+\dim Z<\dim X$, in which case $\copyright(Z,f)=\Hom(Z,X)$;
\item If $\dim(\Im f)=\dim X$, then $\circledS(Z,f)\cap\copyright(Z,f)=\emptyset$;
\item $\copyright(Z,{\bf1}_X)=\Hom(Z,\partial X)$, where $\partial X$ and $\Hom(Z,\partial X)$ are viewed as polytopal complexes in the sense of  \cite{Hom-complexes};
\end{enumerate}
\end{proposition}

\begin{proof} (a) is straightforward from the observation that, for any $Z\in\CONV$, we have
$$
\copyright(Z,f)=\bigcup_{X'\in\mathcal P(X\setminus f)}\Hom(Z,X')\qquad \big(\subset\Hom(Z,X)\big).
$$

\medskip\noindent(b) is straightforward.

\medskip\noindent(c) When $\dim(\Im f)+\dim Z<\dim X$, any $g\in\copyright(Z,f)$ admits arbitrary small perturbations $g'$ for which $g'(Z)\ \cap\ \inte(\Im f)=\emptyset$. On the other hand, if $\dim(\Im f)+\dim Z\ge\dim X$, there exists $g\in\copyright(Z,f)$ such that $\dim(g(Z))+\dim(\Im f)=\dim X$ and $g(Z)\cap\inte(\Im f)\not=\emptyset$. Then, for all sufficiently small perturbations $g'$ of $g$, we have $g'(Z)\cap\inte(\Im f)\not=\emptyset$.

\medskip\noindent(d) By the alternative description of $\copyright(Z,f)$ in part (b) for the case $\dim(\Im f)=\dim X$, every element $g\in\circledS(Z,f)\cap\copyright(Z,f)$ must satisfy the contradictory conditions $\Im f\subset g(Z)$ and $g(Z)\cap\inte(\Im f)=\emptyset$.

\medskip\noindent(e) is straightfoward.
\end{proof}

\begin{proposition}\label{about-complementing2}
Assume $X,Y,Z\in\POL$ and $f:Y\to X$ is in $\POL$. Then $\copyright(Z,f)$ is a semialgebraic set.
\end{proposition}

\begin{proof}
Consider the evaluation map $\ev:\Hom(Z,X)\times Z\to X$. It is not affine, but \emph{bi-affine}, i.e., upon fixing one component, the map is affine in the other. Let $\pi:\Hom(Z,X)\times Z\to\Hom(Z,X)$ be the projection map. Then 
\begin{align*}
\copyright(Z,f)=\overline{\pi\big((\Hom(Z,X)\times Z)\setminus\ev^{-1}(\Im f)\big)}.
\end{align*}
Since $\ev$ is a degree 2 polynomial map, basic properties of semialgebraic sets (e.g.,  the \emph{Tarski-Seidenberg Theorem}  \cite{Real}) guarantee that $\copyright(Z,f)$ is semialgebraic.
\end{proof}

\begin{remark}\label{nonlinearity} In general,  for $Z$ and $f$ in $\POL$, neither of the semialgebraic sets $\circledS(Z,f)$ and $\copyright(Z,f)$ is described by a Boolean combination of linear inequalities. As an example, consider the two squares in the plane $\RR^2$:
\begin{align*}
&\Box'=\conv\{(a_1,a_2)\ |\ a_1,a_2=\pm\epsilon'\},\\
&\Box''=\conv\{(b_1,b_2)\ |\ b_1,b_2=\pm\epsilon''\},
\end{align*}
where $0<\epsilon'\ll\epsilon''$, i.e., $\epsilon''$ is sufficiently larger than $\epsilon'$. Let $\Delta$ be a triangle. Then neither $\circledS(\Delta,\iota)$ nor $\copyright(\Delta,\iota)$, where $\iota:\Box'\hookrightarrow\Box''$ is the inclusion map, can be described by linear constrains. Without delving into the planar geometry, we only mention that this follows from the strict convexity of the following function, wherever it is defined: 
\begin{align*}
&\text{$x$-coordinate of a point $p$ in the upper edge of $\Box''$}\ \mapsto\\
&\qquad\text{$y$-coordinate of the point $q$ on the right edge of $\Box''$,}\\
&\qquad\qquad\text{such that the upper right corner of $\Box'$ is in $[p,q]$}.
\end{align*}
\end{remark}

\begin{proposition}\label{about-complementing3}
Let $X,Y,Z$ and $f:Y\to X$ be in $\CONV$. If $\dim(\Im f)=\dim X$, then $\partial X$ is a strong deformation retract of $\copyright(Z,f)$. In all other cases, $\copyright(Z,f)$ is contractible.
\end{proposition}

\begin{proof}
Assume $\dim(\Im f)=\dim X$. We think of $X\setminus\inte(\Im f)$ as a subset of $\copyright(Z,f)$ via identifying every point $x\in X\setminus\inte(\Im f)$ with the constant map $g:Z\to X$, $g(Z)=x$. Pick a point $z\in Z$. The homotopy
\begin{align*}
H_t:\copyright(Z,f)&\to\copyright(Z,f),\quad t\in[0,1],\\
g&\mapsto
\big(\text{homothety of}\ \Aff(X)\ \text{with coefficient}\ t\ \text{and}\\
&\ \qquad\text{centered at}\ g(z)\big)\circ g
\end{align*}
makes $X\setminus\inte(\Im f)$ a strong deformation retract of $\copyright(Z,f)$. But $\partial X$ is a strong deformation retract of $X\setminus\inte(\Im f)$ via the polar projection onto the boundary from a point $y\in\inte(\Im f)$. Concatenating the two homotopies, we get the desired deformation retraction.

If $\dim(\Im f)<\dim X$, by identifying every point $x\in X$ with the constant map $g:Z\to X$, $g(Z)=x$, we can think of $X$ as a subset of $\copyright(Z,f)$. The homotopy $\{H_t\}_{[0,1]}$ above makes $X$ itself a deformation retract of $\copyright(Z,f)$, and $X$ is contractible.
\end{proof}

\subsection{Complementarity}

For a polytope $P$, let $\tilde\FF(P)$ denote the poset of all proper faces $F\subset P$, ordered by inclusion. We view $\tilde\FF(P)$ as a finite category. For a map $f:Y\to X$ in $\CONV$, we have the contravariant functor:
\begin{align*}
\copyright(\tilde\FF(P), f):\tilde\FF(P)&\to\COMP,
\\
F&\mapsto\copyright(F,f),\\
(\iota:F\hookrightarrow G)&\mapsto\ 
\bigg({\begin{matrix}
\copyright(G,f)&\to&\copyright(F,f)\\
g&\mapsto&g|_F
\end{matrix}}\bigg).
\end{align*}

\begin{lemma}\label{limit_exists}
The limit $\lim\limits_{\longleftarrow}\copyright(\tilde\FF(P),f)$ exists.
\end{lemma}

\begin{proof}
The limit exists in the bigger category of topological spaces and continuous maps: as a space, $\lim\limits_{\longleftarrow}\copyright(\tilde\FF(P),f)$, identifies with the space, with the open-compact topology, of continuous face-wise affine maps $\gamma:\partial P\to X$, satisfying $\gamma|_F\in\copyright(F,f)$ for every $F\in\tilde\FF(P)$. All one needs to show is that $\lim\limits_{\longleftarrow}\copyright(\tilde\FF(P),f)$ is compact. But this follows from \cite[Theorem 5]{Stone} because, for any two faces $F,G\in\tilde\FF(P)$ with $F\subset G$, the induced map $\copyright(G,f)\to\copyright(F,f)$, which is the restriction map, is closed -- an easy exercise. 
\end{proof}

Assume $f:Y\to X$ is a map in $\CONV$, such that $\dim(\Im f)=\dim X$. Then the subsets $\circledS(Z,f),\copyright(Z,f)\subset\Hom(Z,X)$ are disjoint (Proposition \ref{about-complementing1}). Correspondingly, we denote their union by $\circledS(P,f)\amalg\copyright(P,f)$.

Also, by Lemma \ref{limit_exists}, it makes sense to talk about maps in $\COMP$ to and from $\lim\limits_{\longleftarrow}\copyright(\tilde\FF(P),f)$. 

\begin{theorem}\label{comp-quotient}
Let $f:Y\to X$ in $\CONV$ and $P\in\POL$. Assume $\dim(\Im f)=\dim X\le\dim P$. Then there exists a natural injective affine map
\begin{align*}
\rho:\circledS(P,f)\amalg\copyright(P,f)\to
\lim\limits_{\longleftarrow}\copyright(\tilde\FF(P),f).
\end{align*}
Moreover,
\begin{enumerate}[\rm(a)]
\item
The map $\rho$ is bijective if $P$ is simple, not an $n$-gon with $n\ge4$;
\item 
The map $\rho$ is non-bijective if  $f(Y)\subset\inte(X)$ and there is a vertex $v\in P$, such that (i) the facets through $v$ are simplices, and (ii) $P$ is not a pyramid with apex at $v$.
\end{enumerate}
\end{theorem}

Observe that $P$ has a vertex as in part (b) if $P$ is simplicial, not a simplex.

\begin{proof}
As in the proof of Lemma \ref{limit_exists}, we think of the elements of $\lim\limits_{\longleftarrow}\copyright(\tilde\FF(P),f)$ as the continuous face-wise affine maps $\gamma:\partial P\to X$, satisfying $\gamma|_F\in\copyright(F,f)$ whenever $F\in\tilde\FF(P)$.

The crucial observation, based on the condition $\dim(\Im f)=\dim X\le\dim P$, is the following implication: if a map $\gamma\in\lim\limits_{\longleftarrow}\copyright(\tilde\FF(P),f)$ extends to an affine map $g:P\to X$, then either $f(Y)\subset g(P)$ or $g(P)\cap\inte(f(Y))=\emptyset$. Therefore, the assignment $\rho:g\mapsto g|_{\partial P}$ is an injective affine map with the mentioned source and target. It is injective because two affine maps from $P$ coincide if they agree on $\partial P$.

\medskip\noindent(a) Let $P$ be simple and $\gamma:\partial P\to X$ be a face-wise affine map. We want to show that $\gamma$ extends to a map $g:P\to X$. Pick a vertex $v\in P$. Because $v$ is simple, there exists a unique affine map $g:P\to\Aff(X)$, which agrees with $\gamma$ on the facets $F\subset P$ with $v\in F$. We claim that $\gamma$ and $g$ agree on \emph{all} facets of $P$. Observe that if $\gamma$ and $g$ agree on all facets containing some vertex except possibly one, then, because $P$ is simple, the two maps also agree on the remaining facet. Now, the claim is proved by bringing in one by one the facets of $P$ for which the equality $g=\gamma$ has been verified, starting with the initial set of facets through $v$: in this process, until all facets have been incorporated, there is always a vertex, incident with exactly one new facet.

\medskip\noindent(b) Let $v\in P$ be a vertex with the mentioned property. There is a point $w\in\Aff(P)\setminus P$ such that the faces of $P$, not containing $v$, and the simplices 
$$
\conv(w,\Delta),\quad \Delta\in\text{link}(v),
$$ 
form a simplicial sphere $\Pi$, which does not bound a convex body in $\Aff(P)$. Because the faces of $P$ through $v$ are simplices, there is a map $\gamma':\partial P\to\Pi$, which restricts to (i) affine isomorphisms on the faces, containing $v$, and (ii) the identity maps on the faces, not containing $v$. Since $f(Y)\subset\inte(X)$, there is an injective affine map $\gamma'':\conv(\Pi)\to X\setminus\inte(f(Y))$.  Then $\gamma''\gamma'\in\lim\limits_{\longleftarrow}\copyright(\tilde\FF(P),f)\setminus\Im\rho$.
\end{proof}

\section{New challenges}\label{Challenges}

Here we raise two natural problems on polytopes, motivated by  sandwiching and complementing. One is of classical flavor and the other is more of a research program.

The class of polytopes, for which the map $\rho$ in Theorem \ref{comp-quotient} is bijective, is considerably larger than the class, mentioned in Theorem  \ref{comp-quotient}(a). For instance, if the nonsimple vertices of $P$ are rare compared to  the simple ones, than the same bijectivity argument applies. Explicit examples are provided by the \emph{anti-bipyramids} -- they have two antipodal vertices, separated by the zig-zagging equators through the other vertices, which are all simple. Interestingly, the anti-bipyramids with $4n+2$ facets can be realized as the hom-polytopes $\Hom(P_{2n+1},[0,1])$, where $P_k$ is the regular $k$-gon. 

Call a polytope $P$ \emph{affine-rigid} if every continuous face-wise affine map from $\partial P$ to a vector space extends to an affine map from $P$. The isometric version of this concept is the much studied \emph{rigidity property} in metric polytope theory. The story goes back to the \emph{Cauchy Theorem,} showing that all 3-polytopes are rigid. This was generalized to all higher dimensions by Alexandrov in the mid-20th century \cite{Rigidity}. As the proof of Theorem \ref{comp-quotient} shows, the map $\rho$ is bijective if and only if $P$ is affine-rigid.

\begin{problem}\label{Classical}
Classify the affine-rigid polytopes.
\end{problem}

The functorial approach to polytopes explains how the (bi)functors 
$$
\Hom,\ \otimes,\ \ker^*,\ \coker^*
$$
are \emph{internal} for $\POL$. However, Propositions \ref{about-sanwishing}, \ref{about-complementing2}, and Remark  \ref{nonlinearity}, show that, for the functors $\circledS$ and $\copyright$, one needs to invoke more general semialgebraic sets.

\begin{problem}\label{Motivic}
Is there a reasonably small, finitely complete and co-complete, self-enriched symmetric monoidal extension of $\POL$, which makes the functors $\circledS$ and $\copyright$ internal? Does the category of compact subanalytic sets and affine maps between them have all these properties?
\end{problem}

The category of all compact semialgebraic sets and affine maps does not seem to be the right choice: the conical example in \cite[Theorem 3.15]{Velasco} hints at the existence of compact convex semialgebaic sets $X$ and $Y$, for which $\Hom(X,Y)$ is not semialgebraic.

Notice that, if a category $\mathcal M$ as in Problem \ref{Motivic} exists, then the functor $\copyright(-,f)$ is representable whenever the closure $\overline{\Im f}$ is a full-dimensional topological manifold with boundary: $\copyright(-,f)=\Hom_{\mathcal M}(-,X\setminus\inte(\Im f))$.

\bibliographystyle{plain}
\bibliography{bibliography}

\end{document}